\documentclass{amsart}
\usepackage{amsfonts}

\newtheorem{theorem}{Theorem}
\theoremstyle{plain}

\newtheorem{corollary}{Corollary}

\newtheorem{definition}{Definition}

\newtheorem{proposition}{Proposition}
\newtheorem{remark}{Remark}

\numberwithin{equation}{section}
\input{tcilatex}

\begin{document}
\title[On Some Integral Inequalities via $h-$Convexity]{On Some Integral
Inequalities via $h-$Convexity}
\author{Mevl\"{u}t TUN\c{C}}
\address{Kilis 7 Aral\i k University, Faculty of Art and Sciences,
Department of Mathematics, Kilis, 79000, Turkey.}
\email{mevluttunc@kilis.edu.tr}
\date{January 06, 2012}

\begin{abstract}
In this paper, we establish some new inequalities for class of $SX\left(
h,I\right) $ convex functions which are supermultiplicative or superadditive
and nonnegative. And we also give some applications for special means.
\end{abstract}

\maketitle

\section{Introduction}

The following definition is well known in the literature [\ref{mit2}]: A
function $f:I\rightarrow 
\mathbb{R}
,\emptyset \neq I\subseteq 
\mathbb{R}
,$ is said to be convex on I if inequality

\begin{equation}
f\left( tx+\left( 1-t\right) y\right) \leq tf\left( x\right) +\left(
1-t\right) f\left( y\right)  \label{101}
\end{equation}%
holds for all $x,y\in I$ and $t\in \left[ 0,1\right] $. Geometrically, this
means that if $P,Q$ and $R$ are three distinct points on the graph of $f$
with $Q$ between $P$ and $R$, then $Q$ is on or below chord $PR$.

Let $f:I\subseteq 
\mathbb{R}
\rightarrow 
\mathbb{R}
$ be a convex function and $a,b\in I$ with $a<b$. The following double
inequality:$\ $

\begin{equation}
f\left( \frac{a+b}{2}\right) \leq \frac{1}{b-a}\int_{a}^{b}f\left( x\right)
dx\leq \frac{f\left( a\right) +f\left( b\right) }{2}  \label{102}
\end{equation}%
is known in the literature as Hadamard's inequality for convex function.
Keep in mind that some of the classical inequalities for means can come from
(\ref{102}) for convenient particular selections of the function $f$. If $f$
is concave, this double inequality hold in the inversed way.\bigskip\ 

The inequalities (\ref{102}) which have numerous uses in a variety of
settings, has been came a significant groundwork in mathematical analysis
and optimization. Many reports have provided new proof, extensions and
considering its refinements, generalizations, numerous interpolations and
applications, for example, in the theory of special means and information
theory. For some results on generalizations, extensions and applications of
the Hermite-Hadamard inequalities, see [\ref{dr1}, \ref{god}, \ref{hud}, \ref%
{var}-\ref{oz}].

\begin{definition}
\bigskip \textit{[\ref{god}] We say that }$f:I\rightarrow 
\mathbb{R}
$\textit{\ is Godunova-Levin function or that }$f$\textit{\ belongs to the
class }$Q\left( I\right) $\textit{\ if }$f$\textit{\ is non-negative and for
all }$x,y\in I$\textit{\ and }$t\in \left( 0,1\right) $\textit{\ we have \ \
\ \ \ \ \ \ \ \ \ \ \ }%
\begin{equation}
f\left( tx+\left( 1-t\right) y\right) \leq \frac{f\left( x\right) }{t}+\frac{%
f\left( y\right) }{1-t}.  \label{103}
\end{equation}
\end{definition}

\begin{definition}
\textit{[\ref{dr1}] We say that }$f:I\subseteq 
\mathbb{R}
\rightarrow 
\mathbb{R}
$\textit{\ is a }$P-$\textit{function or that }$f$\textit{\ belongs to the
class }$P\left( I\right) $\textit{\ if }$f$\textit{\ is nonnegative and for
all }$x,y\in I$\textit{\ and }$t\in \left[ 0,1\right] ,$\textit{\ we have}%
\begin{equation}
f\left( tx+\left( 1-t\right) y\right) \leq f\left( x\right) +f\left(
y\right) .  \label{104}
\end{equation}
\end{definition}

\begin{definition}
\textit{[\ref{hud}] Let }$s\in \left( 0,1\right] .$\textit{\ A function }$%
f:\left( 0,\infty \right] \rightarrow \left[ 0,\infty \right] $\textit{\ is
said to be }$s-$\textit{convex in the second sense if \ \ \ \ \ \ \ \ \ \ \
\ }%
\begin{equation}
f\left( tx+\left( 1-t\right) y\right) \leq t^{s}f\left( x\right) +\left(
1-t\right) ^{s}f\left( y\right) ,  \label{105}
\end{equation}%
\textit{for all }$x,y\in \left( 0,b\right] $\textit{\ \ and }$t\in \left[ 0,1%
\right] $\textit{. This class of }$\mathit{s-}$\textit{convex functions is
usually denoted by }$K_{s}^{2}$\textit{.}
\end{definition}

In 1978, Breckner introduced $s-$convex functions as a generalization of
convex functions in [\textit{\ref{bre1}}]. Also, in that work Breckner
proved the important fact that the set valued map is $s-$convex only if the
associated support function is $s-$convex function in [\textit{\ref{bre2}}].
A number of properties and connections with s-convex in the first sense are
discussed in paper [\textit{\ref{hud}}]. Of course, $s-$convexity means just
convexity when $s=1$.

\begin{definition}
\bigskip \textit{[\ref{var}] Let }$h:J\subseteq 
\mathbb{R}
\rightarrow 
\mathbb{R}
$\textit{\ be a positive function . We say that }$f:I\subseteq 
\mathbb{R}
\rightarrow 
\mathbb{R}
$\textit{\ is }$h-$\textit{convex function, or that }$f$\textit{\ belongs to
the class }$SX\left( h,I\right) $\textit{, if }$f$\textit{\ is nonnegative
and for all }$x,y\in I$\textit{\ and }$t\in \left[ 0,1\right] $\textit{\ we
have \ \ \ \ \ \ \ \ \ \ \ \ }%
\begin{equation}
f\left( tx+\left( 1-t\right) y\right) \leq h\left( t\right) f\left( x\right)
+h\left( 1-t\right) f\left( y\right) .  \label{106}
\end{equation}
\end{definition}

\bigskip If inequality (\ref{106}) is reversed, then $f$ is said to be $h-$%
concave, i.e. $f\in SV\left( h,I\right) $. Obviously, if $h\left( t\right)
=t $, then all nonnegative convex functions belong to $SX\left( h,I\right) $%
\ and all nonnegative concave functions belong to $SV\left( h,I\right) $; if 
$h\left( t\right) =\frac{1}{t}$, then $SX\left( h,I\right) =Q\left( I\right) 
$; if $h\left( t\right) =1$, then $SX\left( h,I\right) \supseteq P\left(
I\right) $; and if $h\left( t\right) =t^{s}$, where $s\in \left( 0,1\right) $%
, then $SX\left( h,I\right) \supseteq K_{s}^{2}$.

\begin{definition}
\bigskip\ [\textit{\ref{var}}]A function $h:J\rightarrow 
\mathbb{R}
$ is said to be a supermultiplicative function if%
\begin{equation}
h\left( xy\right) \geq h\left( x\right) h\left( y\right)  \label{107}
\end{equation}%
for all $x,y\in J.$
\end{definition}

If inequality (\ref{107}) is reversed, then $h$ is said to be a
submultiplicative function. If equality held in (\ref{107}), then $h$ is
said to be a multiplicative function.

\begin{definition}
\lbrack \textit{\ref{alzer}}]A function $h:J\rightarrow 
\mathbb{R}
$ is said to be a superadditive function if%
\begin{equation}
h\left( x+y\right) \geq h\left( x\right) +h\left( y\right)  \label{108}
\end{equation}%
for all $x,y\in J.$
\end{definition}

\begin{definition}
\lbrack \textit{\ref{skala}}] Two functions $h:X\rightarrow 
\mathbb{R}
$ and $g:X\rightarrow 
\mathbb{R}
$ are said to be similarly ordered, shortly $f$ s.o. $g$, if%
\begin{equation}
\left( f\left( x\right) -f\left( y\right) \right) \left( g\left( x\right)
-g\left( y\right) \right) \geq 0  \label{109}
\end{equation}%
for every $x,y\in X.$
\end{definition}

\begin{remark}
\lbrack \textit{\ref{var}}]Let $h$ be a non-negative function such that%
\begin{equation}
h\left( \alpha \right) \geq \alpha  \label{110}
\end{equation}%
for all $\alpha \in (0,1)$. For example, the function $h_{k}(x)=x^{k}$ where 
$k\leq 1$ and $x>0$ has that property. If $f$ is a non-negative convex
function on $I$ , then for $x,y\in I$ , $\alpha \in (0,1)$ we have 
\begin{equation}
f\left( \alpha x+(1-\alpha )y\right) \leq \alpha f(x)+(1-\alpha )f(y)\leq
h(\alpha )f(x)+h(1-\alpha )f(y).  \label{111}
\end{equation}%
So, $f\in SX(h,I)$. Similarly, if the function $h$ has the property: $%
h(\alpha )\leq \alpha $ for all $\alpha \in (0,1)$, then any non-negative
concave function $f$ belongs to the class $SV(h,I)$.
\end{remark}

\begin{proposition}
\lbrack \textit{\ref{var}}]Let $f$ and $g$ be a similarly ordered functions
on $I$ , i.e.%
\begin{equation}
\left( f(x)-f(y)\right) \left( g(x)-g(y)\right) \geq 0,  \label{112}
\end{equation}%
for all $x,y\in I$. If $f\in SX\left( h_{1},I\right) $, $g\in SX\left(
h_{2},I\right) $ and $h(\alpha )+h(1-\alpha )\leq c$ for all $\alpha \in
(0,1)$, where $h(t)=max\{h_{1}(t),h_{2}(t)\}$ and $c$ is a fixed positive
number, then the product $fg$ belongs to $SX(ch,I)$. If $f$ and $g$ are
oppositely ordered, $f\in SV(h_{1},I)$, $g\in SV(h_{2},I)$ and $h(\alpha
)+h(1-\alpha )\geq c$ for all $\alpha \in (0,1)$, where $h(t)=min%
\{h_{1}(t),h_{2}(t)\}$ and $c>0$, then the product $fg$ belongs to $SV(ch,I)$%
.
\end{proposition}

\bigskip

Up until now, there are many reports on the two convex function, two $s-$%
convex functions, two $m-$convex functions or on the product of the $s-$%
convex function with an ordinary convex function. And in this study, in
addition to its predecessors, the new inequalities on the product of classes
of $h-$convex function will be obtained by using the elementary analysis and
the applications in the special means for the obtained inequalities will be
provided. In this paper we will imply $M(a,b)=f(a)g(a)+f(b)g(b)$ and $%
N(a,b)=f(a)g(b)+f(b)g(a).$

\section{\protect\bigskip Main Results}

The following inequalities is well known in the literature; For $\mu \leq
\lambda $ and $\kappa \leq \varepsilon $ and $\mu ,\lambda ,\kappa
,\varepsilon \in 
\mathbb{R}
,$%
\begin{eqnarray}
\mu \varepsilon +\lambda \kappa &\leq &\mu \kappa +\lambda \varepsilon
\label{*} \\
\mu \kappa &\leq &\lambda \varepsilon  \label{**}
\end{eqnarray}%
The inequality (\ref{*}) is more useful than the inequality (\ref{**}). In
our proofs, the inequality (\ref{*}) will be used.

\begin{theorem}
\label{th1}\bigskip Let $f,g\in SX(h,I)$, $h$ is super-multiplicative and $%
f,g$ be similarly ordered functions on $I$ for all $x,y\in I\subseteq 
\mathbb{R}
,$ and $\left( f,g\right) \left( x\right) \geq x$ and $h\left( t\right) \geq
t$. Then for all $t\in \left[ 0,1\right] ,$ we have;%
\begin{eqnarray}
&&\frac{2a+3b}{6}\left( f\left( a\right) +g\left( a\right) \right) +\frac{%
3a+2b}{6}\left( f\left( b\right) +g\left( b\right) \right)  \notag \\
&\leq &\frac{1}{b-a}\int_{a}^{b}\left( fg\right) \left( x\right) dx+\left(
fg\right) \left( a\right) \int_{0}^{1}\left[ h\left( t\left( 1-t\right)
\right) +h\left( t^{2}\right) \right] dt  \notag \\
&&+\left( fg\right) \left( b\right) \int_{0}^{1}\left[ h\left( t\left(
1-t\right) \right) +h\left( \left( 1-t\right) ^{2}\right) \right] dt.
\label{201}
\end{eqnarray}
\end{theorem}

\begin{proof}
Since $f,$ $g$ are $h-$convex functions on $I$, we have 
\begin{eqnarray*}
f\left( \alpha x+\beta y\right) &\leq &h\left( \alpha \right) f\left(
x\right) +h\left( \beta \right) f\left( y\right) \\
g\left( \alpha x+\beta y\right) &\leq &h\left( \alpha \right) g\left(
x\right) +h\left( \beta \right) g\left( y\right)
\end{eqnarray*}%
for all $\alpha ,\beta \in (0,1),$ $\alpha +\beta =1$. Using the elementary
inequality $\mu \leq \lambda $ and $\kappa \leq \varepsilon $ then $\mu
\varepsilon +\lambda \kappa \leq \mu \kappa +\lambda \varepsilon $ for $\mu
,\lambda ,\kappa ,\varepsilon \in 
\mathbb{R}
,$ so we get%
\begin{eqnarray}
&&f\left( \alpha x+\beta y\right) \left[ h\left( \alpha \right) g\left(
x\right) +h\left( \beta \right) g\left( y\right) \right]  \notag \\
&&+g\left( \alpha x+\beta y\right) \left[ h\left( \alpha \right) f\left(
x\right) +h\left( \beta \right) f\left( y\right) \right]  \notag \\
&\leq &f\left( \alpha x+\beta y\right) g\left( \alpha x+\beta y\right) 
\notag \\
&&+\left[ h\left( \alpha \right) f\left( x\right) +h\left( \beta \right)
f\left( y\right) \right] \left[ h\left( \alpha \right) g\left( x\right)
+h\left( \beta \right) g\left( y\right) \right]  \label{202}
\end{eqnarray}%
Using the other properties of $f,g$ and $h$ in Theorem \ref{th1}, we get%
\begin{eqnarray*}
&&f\left( \alpha x+\beta y\right) \left[ h\left( \alpha \right) g\left(
x\right) +h\left( \beta \right) g\left( y\right) \right] \\
&&+g\left( \alpha x+\beta y\right) \left[ h\left( \alpha \right) f\left(
x\right) +h\left( \beta \right) f\left( y\right) \right] \\
&=&h\left( \alpha \right) g\left( x\right) f\left( \alpha x+\beta y\right)
+h\left( \beta \right) g\left( y\right) f\left( \alpha x+\beta y\right) \\
&&+h\left( \alpha \right) f\left( x\right) g\left( \alpha x+\beta y\right)
+h\left( \beta \right) f\left( y\right) g\left( \alpha x+\beta y\right) ,
\end{eqnarray*}%
since $f,$ $g$ and $h$ are nonnegative functions%
\begin{eqnarray}
&&h\left( \alpha \right) \left[ g\left( x\right) f\left( \alpha x+\beta
y\right) +f\left( x\right) g\left( \alpha x+\beta y\right) \right]  \notag \\
&&+h\left( \beta \right) \left[ g\left( y\right) f\left( \alpha x+\beta
y\right) +f\left( y\right) g\left( \alpha x+\beta y\right) \right]  \notag \\
&\geq &\alpha \left[ \left( \alpha x+\beta y\right) g\left( x\right) +\left(
\alpha x+\beta y\right) f\left( x\right) \right]  \notag \\
&&+\beta \left[ \left( \alpha x+\beta y\right) g\left( y\right) +\left(
\alpha x+\beta y\right) f\left( y\right) \right]  \notag \\
&=&\alpha \left[ \left( \alpha x+\beta y\right) \left( g\left( x\right)
+f\left( x\right) \right) \right] +\beta \left[ \left( \alpha x+\beta
y\right) \left( g\left( y\right) +f\left( y\right) \right) \right] .
\label{203}
\end{eqnarray}%
For the right hand side of (\ref{202}), again using the other properties of $%
f,g$ and $h$ in Theorem \ref{th1}, we can write%
\begin{eqnarray}
&&f\left( \alpha x+\beta y\right) g\left( \alpha x+\beta y\right)  \notag \\
&&+\left[ h\left( \alpha \right) f\left( x\right) +h\left( \beta \right)
f\left( y\right) \right] \left[ h\left( \alpha \right) g\left( x\right)
+h\left( \beta \right) g\left( y\right) \right]  \notag \\
&=&f\left( \alpha x+\beta y\right) g\left( \alpha x+\beta y\right)
+h^{2}\left( \alpha \right) f\left( x\right) g\left( x\right) +h^{2}\left(
\beta \right) f\left( y\right) g\left( y\right)  \notag \\
&&+h\left( \alpha \right) h\left( \beta \right) f\left( x\right) g\left(
y\right) +h\left( \alpha \right) h\left( \beta \right) f\left( y\right)
g\left( x\right)  \notag \\
&\leq &f\left( \alpha x+\beta y\right) g\left( \alpha x+\beta y\right)
+h^{2}\left( \alpha \right) f\left( x\right) g\left( x\right) +h^{2}\left(
\beta \right) f\left( y\right) g\left( y\right)  \notag \\
&&+h\left( \alpha \beta \right) f\left( x\right) g\left( y\right) +h\left(
\alpha \beta \right) f\left( y\right) g\left( x\right)  \notag \\
&=&f\left( \alpha x+\beta y\right) g\left( \alpha x+\beta y\right)
+h^{2}\left( \alpha \right) f\left( x\right) g\left( x\right) +h^{2}\left(
\beta \right) f\left( y\right) g\left( y\right)  \notag \\
&&+h\left( \alpha \beta \right) \left( f\left( x\right) g\left( y\right)
+f\left( y\right) g\left( x\right) \right)  \notag \\
&\leq &f\left( \alpha x+\beta y\right) g\left( \alpha x+\beta y\right)
+h^{2}\left( \alpha \right) f\left( x\right) g\left( x\right) +h^{2}\left(
\beta \right) f\left( y\right) g\left( y\right)  \notag \\
&&+h\left( \alpha \beta \right) \left( f\left( x\right) g\left( x\right)
+f\left( y\right) g\left( y\right) \right)  \notag \\
&=&f\left( \alpha x+\beta y\right) g\left( \alpha x+\beta y\right) +\left[
h^{2}\left( \alpha \right) +h\left( \alpha \beta \right) \right] f\left(
x\right) g\left( x\right)  \notag \\
&&+\left[ h^{2}\left( \beta \right) +h\left( \alpha \beta \right) \right]
f\left( y\right) g\left( y\right)  \label{204}
\end{eqnarray}%
Now by combining expression (\ref{203}) and (\ref{204}), we obtain%
\begin{eqnarray}
&&\alpha \left[ \left( \alpha x+\beta y\right) \left( g\left( x\right)
+f\left( x\right) \right) \right] +\beta \left[ \left( \alpha x+\beta
y\right) \left( g\left( y\right) +f\left( y\right) \right) \right]  \notag \\
&\leq &f\left( \alpha x+\beta y\right) g\left( \alpha x+\beta y\right) + 
\left[ h^{2}\left( \alpha \right) +h\left( \alpha \beta \right) \right]
f\left( x\right) g\left( x\right)  \label{205} \\
&&+\left[ h^{2}\left( \beta \right) +h\left( \alpha \beta \right) \right]
f\left( y\right) g\left( y\right) .  \notag
\end{eqnarray}%
If we choose $x=a,$ $y=b$ and $\beta =1-\alpha $ in (\ref{205}), we have%
\begin{eqnarray}
&&\left( \alpha ^{2}a+\left( 1-\alpha \right) b\right) \left( f\left(
a\right) +g\left( a\right) \right) +\left( \alpha a+\left( 1-\alpha \right)
^{2}b\right) \left( f\left( b\right) +g\left( b\right) \right)  \notag \\
&\leq &f\left( \alpha a+\left( 1-\alpha \right) b\right) g\left( \alpha
a+\left( 1-\alpha \right) b\right) +\left[ h^{2}\left( \alpha \right)
+h\left( \alpha \left( 1-\alpha \right) \right) \right] f\left( a\right)
g\left( a\right)  \notag \\
&&+\left[ h^{2}\left( 1-\alpha \right) +h\left( \alpha \left( 1-\alpha
\right) \right) \right] f\left( b\right) g\left( b\right)  \notag \\
&&\text{by multiplicativity of }h,\text{ we deduce}  \notag \\
&\leq &f\left( \alpha a+\left( 1-\alpha \right) b\right) g\left( \alpha
a+\left( 1-\alpha \right) b\right) +\left[ h\left( \alpha ^{2}\right)
+h\left( \alpha -\alpha ^{2}\right) \right] f\left( a\right) g\left( a\right)
\notag \\
&&+\left[ h\left( \left( 1-\alpha \right) ^{2}\right) +h\left( \alpha
-\alpha ^{2}\right) \right] f\left( b\right) g\left( b\right)  \label{206}
\end{eqnarray}%
By integrating the result with respect to $\alpha $ over $\left[ 0,1\right] $%
, we obtain%
\begin{eqnarray*}
&&\left( a\int_{0}^{1}\alpha ^{2}d\alpha +b\int_{0}^{1}\left( 1-\alpha
\right) d\alpha \right) \left( f\left( a\right) +g\left( a\right) \right) \\
&&+\left( a\int_{0}^{1}\alpha d\alpha +b\int_{0}^{1}\left( 1-\alpha \right)
^{2}d\alpha \right) \left( f\left( b\right) +g\left( b\right) \right) \\
&=&\frac{2a+3b}{6}\left( f\left( a\right) +g\left( a\right) \right) +\frac{%
3a+2b}{6}\left( f\left( b\right) +g\left( b\right) \right) \\
&\leq &\frac{1}{b-a}\int_{a}^{b}fg\left( x\right) dx+f\left( a\right)
g\left( a\right) \int_{0}^{1}\left[ h\left( \alpha ^{2}\right) +h\left(
\alpha -\alpha ^{2}\right) \right] d\alpha \\
&&+f\left( b\right) g\left( b\right) \int_{0}^{1}\left[ h\left( \left(
1-\alpha \right) ^{2}\right) +h\left( \alpha -\alpha ^{2}\right) \right]
d\alpha
\end{eqnarray*}%
which completes the proof.
\end{proof}

\begin{corollary}
\bigskip If in (\ref{201}) we take $h(t)=1$, then we get an integral
inequality for $P-$functions with launching of necessary mathematical
operations,%
\begin{eqnarray}
&&\frac{2a+3b}{6}\left( f\left( a\right) +g\left( a\right) \right) +\frac{%
3a+2b}{6}\left( f\left( b\right) +g\left( b\right) \right)  \notag \\
&\leq &\frac{1}{b-a}\int_{a}^{b}\left( fg\right) \left( x\right) dx+2M\left(
a,b\right) .  \notag
\end{eqnarray}
\end{corollary}

\begin{corollary}
\bigskip If in (\ref{201}) we take $h(t)=t$, then we get an integral
inequality for ordinary convex functions with launching of necessary
mathematical operations,%
\begin{eqnarray*}
&&\frac{2a+3b}{6}\left( f\left( a\right) +g\left( a\right) \right) +\frac{%
3a+2b}{6}\left( f\left( b\right) +g\left( b\right) \right) \\
&\leq &\frac{1}{b-a}\int_{a}^{b}\left( fg\right) \left( x\right) dx+\left(
fg\right) \left( a\right) \int_{0}^{1}\left[ \left( t-t^{2}\right) +\left(
t^{2}\right) \right] dt \\
&&+\left( fg\right) \left( b\right) \int_{0}^{1}\left[ \left( t-t^{2}\right)
+\left( \left( 1-t\right) ^{2}\right) \right] dt \\
&=&\frac{1}{b-a}\int_{a}^{b}\left( fg\right) \left( x\right) dx+\frac{%
M\left( a,b\right) }{2}.
\end{eqnarray*}
\end{corollary}

\begin{corollary}
\bigskip If in (\ref{201}) we take $h(t)=t^{s}$, then we obtain an integral
inequality for $s-$convex functions in the second sense with use of the Beta
function of Euler type%
\begin{eqnarray*}
&&\frac{2a+3b}{6}\left( f\left( a\right) +g\left( a\right) \right) +\frac{%
3a+2b}{6}\left( f\left( b\right) +g\left( b\right) \right) \\
&\leq &\frac{1}{b-a}\int_{a}^{b}\left( fg\right) \left( x\right) dx+\left(
fg\right) \left( a\right) \int_{0}^{1}\left[ t^{s}\left( 1-t\right)
^{s}+t^{2s}\right] dt \\
&&+\left( fg\right) \left( b\right) \int_{0}^{1}\left[ t^{s}\left(
1-t\right) ^{s}+\left( 1-t\right) ^{2s}\right] dt \\
&=&\frac{1}{b-a}\int_{a}^{b}\left( fg\right) \left( x\right) dx+\left(
fg\right) \left( a\right) \left[ \beta \left( s+1,s+1\right) +\beta \left(
2s+1,1\right) \right] \\
&&+\left( fg\right) \left( b\right) \left[ \beta \left( s+1,s+1\right)
+\beta \left( 1,2s+1\right) \right] \\
&=&\frac{1}{b-a}\int_{a}^{b}\left( fg\right) \left( x\right) dx+M\left(
a,b\right) \left[ \beta \left( s+1,s+1\right) +\beta \left( 2s+1,1\right) %
\right] .
\end{eqnarray*}
\end{corollary}

\begin{theorem}
\label{th2}\bigskip Let $f,g\in SX(h,I)$, $h$ is superadditive and
nonnegative such that $h\left( \alpha \right) \geq \alpha $ and $f,g$ be
similarly ordered functions on $I$ for all $x,y\in I\subseteq 
\mathbb{R}
$. Then for all $\alpha \in (0,1)$ and $\alpha +\beta =1$ we have following
inequality;%
\begin{eqnarray}
&&\frac{M\left( a,b\right) }{6}+\frac{N\left( a,b\right) }{3}  \notag \\
&\leq &h\left( 1\right) \left[ f\left( a\right) g\left( a\right)
\int_{0}^{1}h\left( \alpha \right) d\alpha +f\left( b\right) g\left(
b\right) \int_{0}^{1}h\left( 1-\alpha \right) d\alpha \right] .  \label{209}
\end{eqnarray}
\end{theorem}

\begin{proof}
\bigskip Since $f,$ $g$ are $h-$convex functions on $I$, and using right
hand side of (\ref{111}), we have 
\begin{eqnarray*}
\alpha f\left( x\right) +\beta f\left( y\right) &\leq &h\left( \alpha
\right) f\left( x\right) +h\left( \beta \right) f\left( y\right) \\
\alpha g\left( x\right) +\beta g\left( y\right) &\leq &h\left( \alpha
\right) g\left( x\right) +h\left( \beta \right) g\left( y\right)
\end{eqnarray*}%
for all $\alpha ,\beta \in (0,1),$ $\alpha +\beta =1$. Using the elementary
inequality $\mu \leq \lambda $ and $\kappa \leq \varepsilon $ then $\mu
\varepsilon +\lambda \kappa \leq \mu \kappa +\lambda \varepsilon $ for $\mu
,\lambda ,\kappa ,\varepsilon \in 
\mathbb{R}
,$ so we get%
\begin{eqnarray}
&&\left[ \alpha f\left( x\right) +\beta f\left( y\right) \right] \left[
h\left( \alpha \right) g\left( x\right) +h\left( \beta \right) g\left(
y\right) \right]  \notag \\
&&+\left[ \alpha g\left( x\right) +\beta g\left( y\right) \right] \left[
h\left( \alpha \right) f\left( x\right) +h\left( \beta \right) f\left(
y\right) \right]  \notag \\
&\leq &\left[ \alpha f\left( x\right) +\beta f\left( y\right) \right] \left[
\alpha g\left( x\right) +\beta g\left( y\right) \right]  \label{210} \\
&&+\left[ h\left( \alpha \right) f\left( x\right) +h\left( \beta \right)
f\left( y\right) \right] \left[ h\left( \alpha \right) g\left( x\right)
+h\left( \beta \right) g\left( y\right) \right] .  \notag
\end{eqnarray}%
Using the other properties of $h$ in Theorem \ref{th2} on the left hand side
of (\ref{210}), we get%
\begin{eqnarray*}
&&\left[ \alpha f\left( x\right) +\beta f\left( y\right) \right] \left[
h\left( \alpha \right) g\left( x\right) +h\left( \beta \right) g\left(
y\right) \right] \\
&&+\left[ \alpha g\left( x\right) +\beta g\left( y\right) \right] \left[
h\left( \alpha \right) f\left( x\right) +h\left( \beta \right) f\left(
y\right) \right] \\
&=&\alpha h\left( \alpha \right) f\left( x\right) g\left( x\right) +\alpha
h\left( \beta \right) f\left( x\right) g\left( y\right) \\
&&+\beta h\left( \alpha \right) f\left( y\right) g\left( x\right) +\beta
h\left( \beta \right) f\left( y\right) g\left( y\right) \\
&&+\alpha h\left( \alpha \right) f\left( x\right) g\left( x\right) +\alpha
h\left( \beta \right) f\left( y\right) g\left( x\right) \\
&&+\beta h\left( \alpha \right) f\left( x\right) g\left( y\right) +\beta
h\left( \beta \right) f\left( y\right) g\left( y\right) \\
&=&2\alpha h\left( \alpha \right) f\left( x\right) g\left( x\right) +\alpha
h\left( \beta \right) \left[ f\left( x\right) g\left( y\right) +f\left(
y\right) g\left( x\right) \right] \\
&&+\beta h\left( \alpha \right) \left[ f\left( x\right) g\left( y\right)
+f\left( y\right) g\left( x\right) \right] \\
&&+2\beta h\left( \beta \right) f\left( y\right) g\left( y\right) ,
\end{eqnarray*}%
if $h$ be a non-negative function that $h\left( \alpha \right) \geq \alpha $%
\begin{eqnarray}
&&2\alpha h\left( \alpha \right) f\left( x\right) g\left( x\right) +2\beta
h\left( \beta \right) f\left( y\right) g\left( y\right)  \notag \\
&&+\left[ \alpha h\left( \beta \right) +\beta h\left( \alpha \right) \right] %
\left[ f\left( x\right) g\left( y\right) +f\left( y\right) g\left( x\right) %
\right]  \notag \\
&\geq &2\alpha ^{2}f\left( x\right) g\left( x\right) +2\alpha \beta \left[
f\left( x\right) g\left( y\right) +f\left( y\right) g\left( x\right) \right]
+2\beta ^{2}f\left( y\right) g\left( y\right) .  \label{211}
\end{eqnarray}%
However, using the other properties of $f,$ $g$ and $h$ in Theorem \ref{th2}
on the right hand side of (\ref{210}), we get%
\begin{eqnarray}
&&\left[ \alpha f\left( x\right) +\beta f\left( y\right) \right] \left[
\alpha g\left( x\right) +\beta g\left( y\right) \right]  \notag \\
&&+\left[ h\left( \alpha \right) f\left( x\right) +h\left( \beta \right)
f\left( y\right) \right] \left[ h\left( \alpha \right) g\left( x\right)
+h\left( \beta \right) g\left( y\right) \right]  \notag \\
&=&\alpha ^{2}f\left( x\right) g\left( x\right) +\alpha \beta f\left(
x\right) g\left( y\right)  \notag \\
&&+\alpha \beta f\left( y\right) g\left( x\right) +\beta ^{2}f\left(
y\right) g\left( y\right)  \notag \\
&&+h^{2}\left( \alpha \right) f\left( x\right) g\left( x\right) +h\left(
\alpha \right) h\left( \beta \right) f\left( x\right) g\left( y\right) 
\notag \\
&&+h\left( \alpha \right) h\left( \beta \right) f\left( y\right) g\left(
x\right) +h^{2}\left( \beta \right) f\left( y\right) g\left( y\right)  \notag
\\
&=&\alpha ^{2}f\left( x\right) g\left( x\right) +\beta ^{2}f\left( y\right)
g\left( y\right)  \notag \\
&&+\alpha \beta \left[ f\left( x\right) g\left( y\right) +f\left( y\right)
g\left( x\right) \right]  \notag \\
&&+h^{2}\left( \alpha \right) f\left( x\right) g\left( x\right) +h^{2}\left(
\beta \right) f\left( y\right) g\left( y\right)  \notag \\
&&+h\left( \alpha \right) h\left( \beta \right) \left[ f\left( x\right)
g\left( y\right) +f\left( y\right) g\left( x\right) \right]  \notag \\
&\leq &\alpha ^{2}f\left( x\right) g\left( x\right) +\beta ^{2}f\left(
y\right) g\left( y\right) +\alpha \beta \left[ f\left( x\right) g\left(
x\right) +f\left( y\right) g\left( y\right) \right]  \notag \\
&&+h^{2}\left( \alpha \right) f\left( x\right) g\left( x\right) +h^{2}\left(
\beta \right) f\left( y\right) g\left( y\right)  \notag \\
&&+h\left( \alpha \right) h\left( \beta \right) \left[ f\left( x\right)
g\left( x\right) +f\left( y\right) g\left( y\right) \right]  \notag \\
&=&\left[ \alpha ^{2}+\alpha \beta +h^{2}\left( \alpha \right) +h\left(
\alpha \right) h\left( \beta \right) \right] f\left( x\right) g\left(
x\right)  \notag \\
&&+\left[ \alpha \beta +\beta ^{2}+h\left( \alpha \right) h\left( \beta
\right) +h^{2}\left( \beta \right) \right] f\left( y\right) g\left( y\right)
\notag \\
&\leq &\left[ \alpha ^{2}+\alpha \beta +h\left( \alpha \right) h\left(
\alpha +\beta \right) \right] f\left( x\right) g\left( x\right)  \notag \\
&&+\left[ \alpha \beta +\beta ^{2}+h\left( \beta \right) h\left( \alpha
+\beta \right) \right] f\left( y\right) g\left( y\right)  \notag \\
&=&\left[ \alpha ^{2}+\alpha \beta +h\left( \alpha \right) h\left( 1\right) %
\right] f\left( x\right) g\left( x\right)  \label{212} \\
&&+\left[ \alpha \beta +\beta ^{2}+h\left( \beta \right) h\left( 1\right) %
\right] f\left( y\right) g\left( y\right) .  \notag
\end{eqnarray}%
Then by expression (\ref{211}) and (\ref{212});%
\begin{eqnarray}
&&\alpha ^{2}f\left( x\right) g\left( x\right) +2\alpha \beta \left[ f\left(
x\right) g\left( y\right) +f\left( y\right) g\left( x\right) \right] +\beta
^{2}f\left( y\right) g\left( y\right)  \notag \\
&\leq &\left[ \alpha \beta +h\left( \alpha \right) h\left( 1\right) \right]
f\left( x\right) g\left( x\right) +\left[ \alpha \beta +h\left( \beta
\right) h\left( 1\right) \right] f\left( y\right) g\left( y\right)
\label{213}
\end{eqnarray}%
by taking $x=a$, $y=b$ and $\beta =1-\alpha $ in (\ref{213}), we have%
\begin{eqnarray*}
&&\alpha ^{2}f\left( a\right) g\left( a\right) +2\alpha \left( 1-\alpha
\right) \left[ f\left( a\right) g\left( b\right) +f\left( b\right) g\left(
a\right) \right] +\left( 1-\alpha \right) ^{2}f\left( b\right) g\left(
b\right) \\
&\leq &\left[ \alpha \left( 1-\alpha \right) +h\left( \alpha \right) h\left(
1\right) \right] f\left( a\right) g\left( a\right) +\left[ \alpha \left(
1-\alpha \right) +h\left( 1-\alpha \right) h\left( 1\right) \right] f\left(
b\right) g\left( b\right)
\end{eqnarray*}%
By integrating the result with respect to $\alpha $ over $[0,1],$ we get%
\begin{eqnarray*}
&&f\left( a\right) g\left( a\right) \int_{0}^{1}\alpha ^{2}d\alpha +2\left[
f\left( a\right) g\left( b\right) +f\left( b\right) g\left( a\right) \right]
\int_{0}^{1}\alpha \left( 1-\alpha \right) d\alpha \\
&&+f\left( b\right) g\left( b\right) \int_{0}^{1}\left( 1-\alpha \right)
^{2}d\alpha \\
&=&\frac{f\left( a\right) g\left( a\right) +f\left( b\right) g\left(
b\right) }{3}+\frac{f\left( a\right) g\left( b\right) +f\left( b\right)
g\left( a\right) }{3} \\
&\leq &f\left( a\right) g\left( a\right) \left[ \int_{0}^{1}\alpha \left(
1-\alpha \right) d\alpha +h\left( 1\right) \int_{0}^{1}h\left( \alpha
\right) d\alpha \right] \\
&&+f\left( b\right) g\left( b\right) \left[ \int_{0}^{1}\alpha \left(
1-\alpha \right) d\alpha +h\left( 1\right) \int_{0}^{1}h\left( 1-\alpha
\right) dt\right] \\
&=&\frac{f\left( a\right) g\left( a\right) +f\left( b\right) g\left(
b\right) }{6} \\
&&+h\left( 1\right) \left[ f\left( a\right) g\left( a\right)
\int_{0}^{1}h\left( \alpha \right) d\alpha +f\left( b\right) g\left(
b\right) \int_{0}^{1}h\left( 1-\alpha \right) d\alpha \right]
\end{eqnarray*}%
which completes the proof of (\ref{209}).
\end{proof}

\begin{corollary}
\bigskip In Theorem \ref{th2}, if we choose $h(t)=t$, then inequality of\ (%
\ref{209}) brings inequality (\ref{112}) down.
\end{corollary}

\begin{corollary}
In Theorem \ref{th2}, if we choose $h(t)=t^{s}$, then we get an integral
inequality for $s-$convex functions in the second sense;%
\begin{equation}
\frac{M\left( a,b\right) }{6}+\frac{N\left( a,b\right) }{3}\leq \frac{%
M\left( a,b\right) }{s+1}.  \label{214}
\end{equation}%
And, in (\ref{214}), if we choose $s=1,$ then the inequality\ of\ (\ref{214}%
) above brings inequality (\ref{112}) down.
\end{corollary}

\begin{theorem}
\label{th3}\bigskip Let $f,g\in SX\left( h,I\right) $, $h$ is superadditive
and nonnegative such that $h\left( \alpha \right) \geq \alpha $ and $f,g$ be
similarly ordered functions on $I$ for all $x,y\in I\subseteq 
\mathbb{R}
$. Then for all $\alpha \in (0,1)$ and $\alpha +\beta =1,$ we have%
\begin{eqnarray}
&&\frac{M\left( a,b\right) }{3}+\frac{N\left( a,b\right) }{6}  \label{215} \\
&\leq &h\left( 1\right) \left[ f\left( a\right) g\left( a\right)
\int_{0}^{1}h\left( \alpha \right) d\alpha +f\left( b\right) g\left(
b\right) \int_{0}^{1}h\left( 1-\alpha \right) d\alpha \right] .  \notag
\end{eqnarray}
\end{theorem}

\begin{proof}
\bigskip As in the proof of the inequality (\ref{209}), since $f,$ $g$ are $%
h-$convex functions on $I$, and using right hand side of (\ref{112}), we
have 
\begin{eqnarray*}
\alpha f\left( x\right) +\beta f\left( y\right) &\leq &h\left( \alpha
\right) f\left( x\right) +h\left( \beta \right) f\left( y\right) \\
\alpha g\left( x\right) +\beta g\left( y\right) &\leq &h\left( \alpha
\right) g\left( x\right) +h\left( \beta \right) g\left( y\right)
\end{eqnarray*}%
for all $\alpha ,\beta \in (0,1),$ $\alpha +\beta =1$. Using the elementary
inequality $\mu \leq \lambda $ and $\kappa \leq \varepsilon $ then $\mu
\varepsilon +\lambda \kappa \leq \mu \kappa +\lambda \varepsilon $ for $\mu
,\lambda ,\kappa ,\varepsilon \in 
\mathbb{R}
,$ so we obtain%
\begin{eqnarray}
&&\left[ \alpha f\left( x\right) +\beta f\left( y\right) \right] \left[
h\left( \alpha \right) g\left( x\right) +h\left( \beta \right) g\left(
y\right) \right]  \notag \\
&&+\left[ \alpha g\left( x\right) +\beta g\left( y\right) \right] \left[
h\left( \alpha \right) f\left( x\right) +h\left( \beta \right) f\left(
y\right) \right]  \notag \\
&\leq &\left[ \alpha f\left( x\right) +\beta f\left( y\right) \right] \left[
\alpha g\left( x\right) +\beta g\left( y\right) \right]  \label{216} \\
&&+\left[ h\left( \alpha \right) f\left( x\right) +h\left( \beta \right)
f\left( y\right) \right] \left[ h\left( \alpha \right) g\left( x\right)
+h\left( \beta \right) g\left( y\right) \right] .  \notag
\end{eqnarray}%
Using the other properties of $h$ in Theorem \ref{th3} on the left hand side
of (\ref{216}), we get%
\begin{eqnarray}
&&\left[ \alpha f\left( x\right) +\beta f\left( y\right) \right] \left[
h\left( \alpha \right) g\left( x\right) +h\left( \beta \right) g\left(
y\right) \right]  \notag \\
&&+\left[ \alpha g\left( x\right) +\beta g\left( y\right) \right] \left[
h\left( \alpha \right) f\left( x\right) +h\left( \beta \right) f\left(
y\right) \right]  \notag \\
&\geq &2\alpha ^{2}f\left( x\right) g\left( x\right) +2\alpha \beta \left[
f\left( x\right) g\left( y\right) +f\left( y\right) g\left( x\right) \right]
+2\beta ^{2}f\left( y\right) g\left( y\right)  \label{217}
\end{eqnarray}%
However, using the other properties of $f,$ $g$ and $h$ in Theorem \ref{th3}
on the right hand side of (\ref{216}), we obtain%
\begin{eqnarray}
&&\left[ \alpha f\left( x\right) +\beta f\left( y\right) \right] \left[
\alpha g\left( x\right) +\beta g\left( y\right) \right]  \label{218} \\
&&+\left[ h\left( \alpha \right) f\left( x\right) +h\left( \beta \right)
f\left( y\right) \right] \left[ h\left( \alpha \right) g\left( x\right)
+h\left( \beta \right) g\left( y\right) \right]  \notag \\
&=&\alpha ^{2}f\left( x\right) g\left( x\right) +\beta ^{2}f\left( y\right)
g\left( y\right) +\alpha \beta \left[ f\left( x\right) g\left( y\right)
+f\left( y\right) g\left( x\right) \right]  \notag \\
&&+h^{2}\left( \alpha \right) f\left( x\right) g\left( x\right) +h^{2}\left(
\beta \right) f\left( y\right) g\left( y\right) +h\left( \alpha \right)
h\left( \beta \right) \left[ f\left( x\right) g\left( y\right) +f\left(
y\right) g\left( x\right) \right]  \notag \\
&\leq &\alpha ^{2}f\left( x\right) g\left( x\right) +\beta ^{2}f\left(
y\right) g\left( y\right) +\alpha \beta \left[ f\left( x\right) g\left(
x\right) +f\left( y\right) g\left( y\right) \right]  \notag \\
&&+h^{2}\left( \alpha \right) f\left( x\right) g\left( x\right) +h^{2}\left(
\beta \right) f\left( y\right) g\left( y\right) +h\left( \alpha \right)
h\left( \beta \right) \left[ f\left( x\right) g\left( x\right) +f\left(
y\right) g\left( y\right) \right]  \notag \\
&\leq &h^{2}\left( \alpha \right) f\left( x\right) g\left( x\right)
+h^{2}\left( \beta \right) f\left( y\right) g\left( y\right) +h\left( \alpha
\right) h\left( \beta \right) \left[ f\left( x\right) g\left( x\right)
+f\left( y\right) g\left( y\right) \right]  \notag \\
&&+h^{2}\left( \alpha \right) f\left( x\right) g\left( x\right) +h^{2}\left(
\beta \right) f\left( y\right) g\left( y\right) +h\left( \alpha \right)
h\left( \beta \right) \left[ f\left( x\right) g\left( x\right) +f\left(
y\right) g\left( y\right) \right]  \notag \\
&=&2\left( h^{2}\left( \alpha \right) +h\left( \alpha \right) h\left( \beta
\right) \right) f\left( x\right) g\left( x\right) +2\left( h^{2}\left( \beta
\right) +h\left( \alpha \right) h\left( \beta \right) \right) f\left(
y\right) g\left( y\right)  \notag \\
&\leq &2h\left( \alpha \right) h\left( \alpha +\beta \right) f\left(
x\right) g\left( x\right) +2h\left( \beta \right) h\left( \alpha +\beta
\right) f\left( y\right) g\left( y\right)  \notag \\
&=&2h\left( \alpha +\beta \right) \left[ h\left( \alpha \right) f\left(
x\right) g\left( x\right) +h\left( \beta \right) f\left( y\right) g\left(
y\right) \right]  \notag \\
&=&2h\left( 1\right) \left[ h\left( \alpha \right) f\left( x\right) g\left(
x\right) +h\left( \beta \right) f\left( y\right) g\left( y\right) \right] 
\notag
\end{eqnarray}%
Then by expression (\ref{217}) and (\ref{218});%
\begin{eqnarray}
&&\alpha ^{2}f\left( x\right) g\left( x\right) +\alpha \beta \left[ f\left(
x\right) g\left( y\right) +f\left( y\right) g\left( x\right) \right] +\beta
^{2}f\left( y\right) g\left( y\right)  \notag \\
&\leq &h\left( 1\right) \left[ h\left( \alpha \right) f\left( x\right)
g\left( x\right) +h\left( \beta \right) f\left( y\right) g\left( y\right) %
\right]  \label{219}
\end{eqnarray}%
by taking $x=a$, $y=b$ and\ $\beta =1-\alpha $ in (\ref{219}), we have%
\begin{eqnarray*}
&&t^{2}f\left( a\right) g\left( a\right) +t\left( 1-t\right) \left[ f\left(
a\right) g\left( b\right) +f\left( b\right) g\left( a\right) \right] +\left(
1-t\right) ^{2}f\left( b\right) g\left( b\right) \\
&\leq &h\left( 1\right) \left[ h\left( \alpha \right) f\left( a\right)
g\left( a\right) +h\left( 1-\alpha \right) f\left( b\right) g\left( b\right) %
\right] .
\end{eqnarray*}%
By integrating the result with respect to $\alpha $ over $[0,1],$ we obtain%
\begin{eqnarray*}
&&f\left( a\right) g\left( a\right) \int_{0}^{1}\alpha ^{2}d\alpha +2\left[
f\left( a\right) g\left( b\right) +f\left( b\right) g\left( a\right) \right]
\int_{0}^{1}\alpha \left( 1-\alpha \right) d\alpha \\
&&+f\left( b\right) g\left( b\right) \int_{0}^{1}\left( 1-\alpha \right)
^{2}d\alpha \\
&=&\frac{f\left( a\right) g\left( a\right) +f\left( b\right) g\left(
b\right) }{3}+\frac{f\left( a\right) g\left( b\right) +f\left( b\right)
g\left( a\right) }{6} \\
&\leq &h\left( 1\right) \left[ h\left( \alpha \right) f\left( a\right)
g\left( a\right) +h\left( 1-\alpha \right) f\left( b\right) g\left( b\right) %
\right] \\
&=&h\left( 1\right) \left[ f\left( a\right) g\left( a\right)
\int_{0}^{1}h\left( \alpha \right) d\alpha +f\left( b\right) g\left(
b\right) \int_{0}^{1}h\left( 1-\alpha \right) d\alpha \right]
\end{eqnarray*}%
which completes the proof of (\ref{215}).
\end{proof}

\begin{theorem}
\bigskip \label{th4}\bigskip Let $f,g\in SX(h,I)$, $h$ is superadditive.
Then for all $\alpha \in (0,1),$ we have following inequality;%
\begin{eqnarray}
&&\left[ f\left( \frac{a+b}{2}\right) \frac{g\left( a\right) +g\left(
b\right) }{2}+g\left( \frac{a+b}{2}\right) \frac{f\left( a\right) +f\left(
b\right) }{2}\right]  \notag \\
&&\times \left[ \int_{0}^{1}h\left( \alpha \right) d\alpha
+\int_{0}^{1}h\left( 1-\alpha \right) d\alpha \right]  \notag \\
&\leq &\left( fg\right) \left( \frac{a+b}{2}\right) +h^{2}\left( 1\right)
\left( M\left( a,b\right) +N\left( a,b\right) \right) .  \label{220}
\end{eqnarray}
\end{theorem}

\begin{proof}
\bigskip Since $f,g\in SX\left( h,I\right) ,$ we have%
\begin{eqnarray*}
f\left( \frac{a+b}{2}\right) &=&f\left( \frac{\alpha a+\left( 1-\alpha
\right) b}{2}+\frac{\left( 1-\alpha \right) a+\alpha b}{2}\right) \\
&\leq &\frac{f\left( \alpha a+\left( 1-\alpha \right) b\right) +f\left(
\left( 1-\alpha \right) a+\alpha b\right) }{2} \\
&\leq &\frac{h\left( \alpha \right) f\left( a\right) +h\left( 1-\alpha
\right) f\left( b\right) +h\left( 1-\alpha \right) f\left( a\right) +h\left(
\alpha \right) f\left( b\right) }{2} \\
&=&\frac{\left[ h\left( \alpha \right) +h\left( 1-\alpha \right) \right] %
\left[ f\left( a\right) +f\left( b\right) \right] }{2} \\
g\left( \frac{a+b}{2}\right) &\leq &\frac{\left[ h\left( \alpha \right)
+h\left( 1-\alpha \right) \right] \left[ g\left( a\right) +g\left( b\right) %
\right] }{2}
\end{eqnarray*}%
for all $\alpha \in (0,1)$. Using the elementary inequality $\mu \leq
\lambda $ and $\kappa \leq \varepsilon $ then $\mu \varepsilon +\lambda
\kappa \leq \mu \kappa +\lambda \varepsilon $ for $\mu ,\lambda ,\kappa
,\varepsilon \in 
\mathbb{R}
,$ so we get%
\begin{eqnarray*}
&&f\left( \frac{a+b}{2}\right) \frac{\left[ h\left( \alpha \right) +h\left(
1-\alpha \right) \right] \left[ g\left( a\right) +g\left( b\right) \right] }{%
2} \\
&&+g\left( \frac{a+b}{2}\right) \frac{\left[ h\left( \alpha \right) +h\left(
1-\alpha \right) \right] \left[ f\left( a\right) +f\left( b\right) \right] }{%
2} \\
&=&\left[ f\left( \frac{a+b}{2}\right) \frac{g\left( a\right) +g\left(
b\right) }{2}+g\left( \frac{a+b}{2}\right) \frac{f\left( a\right) +f\left(
b\right) }{2}\right] \left[ h\left( \alpha \right) +h\left( 1-\alpha \right) %
\right] \\
&\leq &f\left( \frac{a+b}{2}\right) g\left( \frac{a+b}{2}\right) +\frac{%
\left[ h\left( \alpha \right) +h\left( 1-\alpha \right) \right] ^{2}\left[
f\left( a\right) +f\left( b\right) \right] \left[ g\left( a\right) +g\left(
b\right) \right] }{4} \\
&\leq &f\left( \frac{a+b}{2}\right) g\left( \frac{a+b}{2}\right) +\frac{%
\left[ h\left( 1\right) \right] ^{2}\left[ f\left( a\right) +f\left(
b\right) \right] \left[ g\left( a\right) +g\left( b\right) \right] }{4},
\end{eqnarray*}%
by integrating the result with respect to $\alpha $ over $[0,1],$ we obtain (%
\ref{220}).
\end{proof}

\begin{theorem}
\bigskip \label{th5}\bigskip Let $f,$ $g\in SX(h,I)$, $h$ is
supermultiplicative and . Then for all $\alpha \in (0,1),$ we have following
inequality;%
\begin{eqnarray}
&&\frac{2f\left( \frac{a+b}{2}\right) g\left( \frac{a+b}{2}\right) }{M\left(
a,b\right) }  \notag \\
&\leq &\left[ \int_{0}^{1}h\left( \alpha ^{2}\right) d\alpha
+2\int_{0}^{1}h\left( \alpha -\alpha ^{2}\right) d\alpha
+\int_{0}^{1}h\left( \left( 1-\alpha \right) ^{2}\right) d\alpha \right] .
\label{221}
\end{eqnarray}
\end{theorem}

\begin{proof}
\bigskip Since $f,g\in SX(h,I)$, we write%
\begin{eqnarray*}
f\left( \frac{a+b}{2}\right) &\leq &\left[ h\left( \alpha \right) +h\left(
1-\alpha \right) \right] \frac{f\left( a\right) +f\left( b\right) }{2} \\
g\left( \frac{a+b}{2}\right) &\leq &\left[ h\left( \alpha \right) +h\left(
1-\alpha \right) \right] \frac{g\left( a\right) +g\left( b\right) }{2}
\end{eqnarray*}%
for all $\alpha \in (0,1)$. Since $h$ is supermultiplicative function and $f$
and $g$ are similarly ordered functions, we get%
\begin{eqnarray*}
&&f\left( \frac{a+b}{2}\right) g\left( \frac{a+b}{2}\right) \\
&\leq &\frac{\left[ h\left( \alpha \right) +h\left( 1-\alpha \right) \right]
^{2}}{4}\left( f\left( a\right) +f\left( b\right) \right) \left( g\left(
a\right) +g\left( b\right) \right) \\
&=&\frac{\left[ h^{2}\left( \alpha \right) +2h\left( \alpha \right) h\left(
1-\alpha \right) +h^{2}\left( 1-\alpha \right) \right] ^{2}}{4} \\
&&\text{x}\left( f\left( a\right) g\left( a\right) +f\left( b\right) g\left(
b\right) +f\left( a\right) g\left( b\right) +f\left( b\right) g\left(
a\right) \right) \\
&\leq &\frac{\left[ h\left( \alpha ^{2}\right) +2h\left( \alpha \left(
1-\alpha \right) \right) +h\left( \left( 1-\alpha \right) ^{2}\right) \right]
^{2}\left( f\left( a\right) g\left( a\right) +f\left( b\right) g\left(
b\right) \right) }{2}
\end{eqnarray*}%
by integrating the result with respect to $\alpha $ over $[0,1],$ we obtain (%
\ref{221}).
\end{proof}

\begin{corollary}
\bigskip If in (\ref{221}) we take $h(\alpha )=1$, then we obtain an
integral inequality for $P-$functions with launching of necessary
mathematical operations,%
\begin{equation*}
f\left( \frac{a+b}{2}\right) g\left( \frac{a+b}{2}\right) \leq 2M\left(
a,b\right)
\end{equation*}
\end{corollary}

\begin{corollary}
\bigskip If in (\ref{221}) we take $h(\alpha )=\alpha $, then we obtain an
integral inequality for ordinary convex functions with launching of
necessary mathematical operations,%
\begin{equation*}
f\left( \frac{a+b}{2}\right) g\left( \frac{a+b}{2}\right) \leq \frac{M\left(
a,b\right) }{2}.
\end{equation*}
\end{corollary}

\begin{corollary}
\bigskip If in (\ref{221}) we take $h(\alpha )=\alpha ^{s}$, then we obtain
an integral inequality for $s-$convex functions in the second sense with the
use of the Beta function of Euler type%
\begin{eqnarray*}
&&f\left( \frac{a+b}{2}\right) g\left( \frac{a+b}{2}\right) \\
&\leq &\frac{M\left( a,b\right) }{2}\left[ \int_{0}^{1}\alpha ^{2s}d\alpha
+2\int_{0}^{1}\alpha ^{s}\left( 1-\alpha \right) ^{s}d\alpha
+\int_{0}^{1}\left( 1-\alpha \right) ^{2s}d\alpha \right] \\
&=&\frac{M\left( a,b\right) }{2}\left\{ \beta \left( 2s+1,1\right) +2\beta
\left( 2s+1,2s+1\right) +\beta \left( 1,2s+1\right) \right\} \\
&=&M\left( a,b\right) \left\{ \beta \left( 2s+1,1\right) +\beta \left(
2s+1,2s+1\right) \right\} .
\end{eqnarray*}
\end{corollary}

\begin{theorem}
\bigskip \bigskip \label{th6}\bigskip Let $I=\left[ a,b\right] \subseteq 
\mathbb{R}
,$ $f,$ $g\in SX(h,I)$, $h$ is supermultiplicative and $f,$ $g$ are
symmetric about $\frac{a+b}{2}$, then for all $x\in I$ ad $\alpha \in \left[
0,1\right] ,$ we have the following inequalities;%
\begin{eqnarray}
&&\frac{1}{b-a}\int_{a}^{b}f\left( x\right) g\left( x\right) dx  \notag \\
&\leq &\frac{M\left( a,b\right) +N\left( a,b\right) }{4}  \label{222} \\
&&\times \left[ \int_{0}^{1}h\left( \alpha ^{2}\right) d\alpha
+2\int_{0}^{1}h\left( \alpha \left( 1-\alpha \right) \right) d\alpha
+\int_{0}^{1}h\left( \left( 1-\alpha \right) ^{2}\right) d\alpha \right] 
\notag
\end{eqnarray}
\end{theorem}

\begin{proof}
\bigskip Since $f,g\in SX(h,I)$, we can write%
\begin{eqnarray*}
f\left( \alpha a+\left( 1-\alpha \right) b\right) +f\left( \left( 1-\alpha
\right) a+\alpha b\right) &\leq &\left[ h\left( \alpha \right) +h\left(
1-\alpha \right) \right] \left[ f\left( a\right) +f\left( b\right) \right] \\
g\left( \alpha a+\left( 1-\alpha \right) b\right) +g\left( \left( 1-\alpha
\right) a+\alpha b\right) &\leq &\left[ h\left( \alpha \right) +h\left(
1-\alpha \right) \right] \left[ g\left( a\right) +g\left( b\right) \right]
\end{eqnarray*}%
for all $\alpha \in \left[ 0,1\right] $. Since $h$ is a supermultiplicative
function, and $f$ and $g$ are symmetric about $\frac{a+b}{2}$, we get%
\begin{eqnarray*}
&&\left[ f\left( \alpha a+\left( 1-\alpha \right) b\right) +f\left( \left(
1-\alpha \right) a+\alpha b\right) \right] \left[ g\left( \alpha a+\left(
1-\alpha \right) b\right) +g\left( \left( 1-\alpha \right) a+\alpha b\right) %
\right] \\
&=&4f\left( \alpha a+\left( 1-\alpha \right) b\right) g\left( \alpha
a+\left( 1-\alpha \right) b\right) \\
&\leq &\left[ h\left( \alpha \right) +h\left( 1-\alpha \right) \right] ^{2}%
\left[ f\left( a\right) +f\left( b\right) \right] \left[ g\left( a\right)
+g\left( b\right) \right] \\
&=&\left[ h^{2}\left( \alpha \right) +2h\left( \alpha \right) h\left(
1-\alpha \right) +h^{2}\left( 1-\alpha \right) \right] \left[ M\left(
a,b\right) +N\left( a,b\right) \right] \\
&\leq &\left[ h\left( \alpha ^{2}\right) +2h\left( \alpha \left( 1-\alpha
\right) \right) +h\left( \left( 1-\alpha \right) ^{2}\right) \right] \left[
M\left( a,b\right) +N\left( a,b\right) \right] .
\end{eqnarray*}%
By integrating the result with respect to $\alpha $ over $[0,1],$ and taking
into account the change of variable $x=\alpha a+\left( 1-\alpha \right) b,$
we obtain%
\begin{eqnarray*}
&&\frac{4}{b-a}\int_{a}^{b}f\left( x\right) g\left( x\right) dx \\
&\leq &\left[ M\left( a,b\right) +N\left( a,b\right) \right] \left[
\int_{0}^{1}h\left( \alpha ^{2}\right) d\alpha +2\int_{0}^{1}h\left( \alpha
\left( 1-\alpha \right) \right) d\alpha +\int_{0}^{1}h\left( \left( 1-\alpha
\right) ^{2}\right) d\alpha \right]
\end{eqnarray*}%
which completes the proof.
\end{proof}

\begin{corollary}
\bigskip If in (\ref{222}) we take $h(\alpha )=1$, then we obtain an
integral inequality for $P-$functions with launching of necessary
mathematical operations,%
\begin{equation*}
\frac{1}{b-a}\int_{a}^{b}f\left( x\right) g\left( x\right) dx\leq M\left(
a,b\right) +N\left( a,b\right) .
\end{equation*}
\end{corollary}

\begin{corollary}
\bigskip \bigskip If in (\ref{222}) we take $h(\alpha )=\alpha $, then we
obtain an integral inequality for ordinary convex functions with the use of
the necessary mathematical operations,%
\begin{equation*}
\frac{1}{b-a}\int_{a}^{b}f\left( x\right) g\left( x\right) dx\leq \frac{%
M\left( a,b\right) +N\left( a,b\right) }{4}.
\end{equation*}
\end{corollary}

\begin{corollary}
\bigskip If in (\ref{222}) we take $h(\alpha )=\alpha ^{s}$, then we obtain
an integral inequality for $s-$convex functions in the second sense the with
use of the Beta function of Euler type\bigskip 
\begin{equation*}
\frac{1}{b-a}\int_{a}^{b}f\left( x\right) g\left( x\right) dx\leq \frac{%
M\left( a,b\right) +N\left( a,b\right) }{2}\left\{ \beta \left(
2s+1,1\right) +\beta \left( 2s+1,2s+1\right) \right\} .
\end{equation*}
\end{corollary}

\section{\protect\bigskip Applications to Some Special Means}

\bigskip We now consider the applications of our Theorems to the following
special means

The arithmetic mean: $A=A\left( a,b\right) :=\frac{a+b}{2},$\ \ $a,b\geq 0,$

The geometric mean: $G=G\left( a,b\right) :=\sqrt{ab},$ \ $a,b\geq 0,$

The quadratic mean: $K=K\left( a,b\right) :=\sqrt{\frac{a^{2}+b^{2}}{2}}$ \ $%
a,b\geq 0,$

The p-logarithmic mean:$L_{p}=L_{p}\left( a,b\right) :=\left\{ 
\begin{array}{l}
\left[ \frac{b^{p+1}-a^{p+1}}{\left( p+1\right) \left( b-a\right) }\right]
^{1/p}\text{ \ \ \ \ \ if \ \ }a\neq b \\ 
a\text{ \ \ \ \ \ \ \ \ \ \ \ \ \ \ \ \ \ \ \ \ \ \ if \ \ }a=b%
\end{array}%
\right. ,$ \ $\ \ \ p\in 
\mathbb{R}
\backslash \left\{ -1,0\right\} ;$ \ $a,b>0.$

\bigskip The following inequality is well known in the literature:%
\begin{equation*}
H\leq G\leq L\leq I\leq A\leq K
\end{equation*}

It is also known that $L_{p}$ is monotonically increasing over $p\in 
\mathbb{R}
,$ denoting $L_{1}=A,$ $L_{0}=I$ and $L_{-1}=L.$

The following propositions holds:\bigskip

\begin{proposition}
\bigskip Let $a,b\in 
\mathbb{R}
,$ $0<a<b$ and $n\in 
\mathbb{Z}
,$ $\left\vert n\right\vert \geq 1$. Then, we have:%
\begin{eqnarray}
&&\frac{4}{3}A\left( a^{n+1},b^{n+1}\right) +2G^{2}\left( a,b\right) A\left(
a^{n-1},b^{n-1}\right)  \notag \\
&\leq &L_{2n}^{2n}\left( a,b\right) +A\left( a^{2n},b^{2n}\right) .
\label{301}
\end{eqnarray}
\end{proposition}

\begin{proof}
If we apply Theorem \ref{th1} for $f(x)=g(x)=x^{n}$, $h\left( \alpha \right)
=\alpha $ where $x\in 
\mathbb{R}
,$ $n\in 
\mathbb{Z}
,$ $\left\vert n\right\vert \geq 1,$\ we get the proof (\ref{301}).
\end{proof}

\begin{proposition}
\bigskip Let $a,b\in 
\mathbb{R}
,$ $0<a<b$ and $n\in 
\mathbb{Z}
,$ $\left\vert n\right\vert \geq 1$. Then, we have:%
\begin{equation}
G^{2}\left( a^{n},b^{n}\right) \leq 2A\left( a^{2n},b^{2n}\right) .
\label{302}
\end{equation}
\end{proposition}

\begin{proof}
$\bigskip $The proof is immediate from Theorem \ref{th2} applied for $%
f(x)=g(x)=x^{n}$, $h\left( \alpha \right) =\alpha $ where $x\in 
\mathbb{R}
,$ $n\in 
\mathbb{Z}
,$ $\left\vert n\right\vert \geq 1$.
\end{proof}

\begin{proposition}
$\bigskip $Let $a,b\in 
\mathbb{R}
,$ $0<a<b.$ Then, we have:%
\begin{equation}
\frac{1}{G^{2}\left( a,b\right) }\leq K\left( a,b\right) .  \label{303}
\end{equation}
\end{proposition}

\begin{proof}
$\bigskip $The assertion follows from Theorem \ref{th2} applied to $%
f(x)=g(x)=\frac{1}{x}$, $x\in \left[ a,b\right] $ and $h\left( \alpha
\right) =\alpha .$\bigskip
\end{proof}

\begin{proposition}
\bigskip Let $a,b\in 
\mathbb{R}
,$ $0<a<b$ and $n\in 
\mathbb{Z}
,$ $\left\vert n\right\vert \geq 1$. Then, we have:%
\begin{equation}
G^{2}\left( a^{n},b^{n}\right) \leq A\left( a^{2n},b^{2n}\right)  \label{304}
\end{equation}
\end{proposition}

\begin{proof}
The assertion follows from Theorem \ref{th3} applied to $f(x)=g(x)=\frac{1}{x%
}$, $x\in \left[ a,b\right] $ and $h\left( \alpha \right) =\alpha .$\bigskip
\end{proof}

\begin{proposition}
$\bigskip $Let $a,b\in 
\mathbb{R}
,$ $0<a<b$ and $n\in 
\mathbb{Z}
,$ $\left\vert n\right\vert \geq 1$. Then, we have:%
\begin{equation}
L_{2n}^{2n}\left( a,b\right) \leq \frac{A\left( a^{2n},b^{2n}\right)
+G^{2}\left( a^{n},b^{n}\right) }{2}.  \label{305}
\end{equation}
\end{proposition}

\begin{proof}
\bigskip If we apply Theorem \ref{th6} for $f(x)=g(x)=x^{n}$, $h\left(
\alpha \right) =\alpha $ where $x\in 
\mathbb{R}
,$ $n\in 
\mathbb{Z}
,$ $\left\vert n\right\vert \geq 1,$\ we get the proof (\ref{305}).
\end{proof}

\begin{proposition}
\bigskip Let $a,b\in 
\mathbb{R}
,$ $0<a<b$ and $n\in 
\mathbb{Z}
,$ $\left\vert n\right\vert \geq 1$. Then, we have:%
\begin{equation}
G\left( a,b\right) \leq A\left( a,b\right)  \label{306}
\end{equation}
\end{proposition}

\begin{proof}
\bigskip If we apply Theorem \ref{th6} for $f(x)=g(x)=\frac{1}{x}$, $x\in %
\left[ a,b\right] $ and $h\left( \alpha \right) =\alpha ,$\ we get the proof
(\ref{306}).
\end{proof}

$\bigskip $

$\bigskip $

\end{document}